\documentclass[a4paper,11pt,oneside]{article} 
\usepackage{amsfonts,amssymb}
\usepackage{color} \usepackage{mathrsfs} \let\mathcal\mathscr
\usepackage[all,ps,cmtip]{xy}
 \usepackage{geometry}

\title{\sc Categorical crepant resolutions for quotient singularities}
\author{\sc Roland Abuaf \footnote{Imperial College London, South Kensington Campus, London SW7 2AZ, United-Kingdom. E-mail :\textit{r.abuaf@imperial.ac.uk}. Supported by EPSRC programme grant $EP/G06170X/1$.}} 
\date{\today}
\usepackage[dvips]{graphicx}

\usepackage{amsfonts,amssymb}
\usepackage{color} \usepackage{mathrsfs} \let\mathcal\mathscr

\usepackage[all,ps,cmtip]{xy}

\DeclareGraphicsExtensions{eps}
\DeclareGraphicsExtensions{pdf}

\usepackage[T1]{fontenc}
\usepackage{amssymb}

\usepackage{amsmath}
\usepackage{latexsym}

\usepackage[latin1]{inputenc}

\newtheorem{theo}{Theorem}[subsection]

\newtheorem{prop}[theo]{Proposition}

\newtheorem{defi}[theo]{Definition}

\newtheorem{cor}[theo]{Corollary}
\newtheorem{conj}[theo]{Conjecture}

\def\DB{\mathrm{D^{b}}}
\def\OO{\mathcal{O}}

\def\DP{\mathrm{D^{perf}}}
\def\DU{\mathrm{D}}

\def\R0{\mathrm{R^{0}}}

\def\Hh{\mathrm{H}om}
\def\H{\mathcal{H}}
\def\LL{\mathrm{\textbf{L}}}

\def\RR{\mathrm{\textbf{R}}}
\def\OO{\mathcal{O}}

\def\d{\delta}

\def\FF{\mathcal{F}}
\def\GG{\mathcal{G}}
\def\G{\mathrm{G}}
\def\T{\mathcal{T}}
\def\AA{\mathcal{A}}
\def\BB{\mathcal{B}}
\def\X{\tilde{X}}
\def\V{\tilde{V}}
\def\XX{\mathcal{X}}

\def\p12{\pi_{{\T_1},{\T_2}}}

\newenvironment{proof}
{
\noindent
\textit{\underline{Proof}} :\\
$\blacktriangleright\;$%
}
{\hspace{\stretch{1}}%
$\blacktriangleleft$}

{
\noindent
\textit{\underline{Proof of the main theorem}} :\\
$\blacktriangleright\;$%
}
{\hspace{\stretch{1}}%
$\blacktriangleleft$}

{
\noindent
\textit{\underline{Proof of theorem 3.2.4}} :\\
$\blacktriangleright\;$%
}
{\hspace{\stretch{1}}%
$\blacktriangleleft$}

{
\noindent
\textit{\underline{Proof of prposition A.0.6}} :\\
$\blacktriangleright\;$%
}
{\hspace{\stretch{1}}%
$\blacktriangleleft$}

{
\noindent
\textit{\underline{Road Map}} :\\
$\blacktriangleright\;$%
}
{\hspace{\stretch{1}}%
$\blacktriangleleft$}

{
\noindent
\textit{\underline{Sketch of the Proof}} :\\
$\blacktriangleright\;$%
}
{\hspace{\stretch{1}}%
$\blacktriangleleft$}

{
\noindent
\textit{\underline{Preuve}} : (id\'ee)\\
$\blacktriangleright\;$%
}
{\hspace{\stretch{1}}%
$\blacktriangleleft$}

\begin{document}

\maketitle
\hspace{10 cm}  \textit{MDD}

\vspace{1,5cm}

\begin{abstract}
We discuss some ``folklore'' results on categorical crepant resolutions for varieties with quotient singularities. 

\end{abstract}

\tableofcontents
\newpage

\begin{section}{Introduction}

Let $X$ be an algebraic variety over $\mathbb{C}$. Hironaka proved that one can find a proper birational morphism $ \tilde{X}
\rightarrow X$, with $\tilde{X}$ smooth. Such an $\tilde{X}$ is called a
\textit{resolution of singularities} of $X$. Unfortunately, given an
algebraic variety $X$, there is, in general, no \textit{minimal resolution} \footnote{that is a resolution which lies under all other resolutions of singularities of $X$.} of singularities of $X$. In case $X$ is
Gorenstein, a crepant resolution of $X$ (that is a resolution $\pi : \tilde{X}
\rightarrow X$ such that $\pi^* \omega_X = \omega_{\tilde{X}}$) is often considered to be
minimal. The conjecture of Bondal-Orlov (see \cite{BO}) gives a precise meaning
to that notion of minimality:

\begin{conj}
 Let $X$ be an algebraic variety with canonical Gorenstein singularities. Assume that $X$ has a crepant resolution of singularities $\tilde{X} \rightarrow X$. Then, for any other resolution of singularities $Y \rightarrow X$, there exists a fully faithful embedding:

\begin{equation*}
\DB (\tilde{X}) \hookrightarrow \DB (Y). 
\end{equation*}
\end{conj}

Varieties admitting a crepant resolution of singularities are quite rare. For instance, non-smooth Gorenstein $\mathbb{Q}$-factorial terminal
singularities (e.g. $\mathbb{C}^{2n}/ \mathbb{Z}_2$, for $n \geq 2$) never
admit crepant resolution of singularities. Thus, it seems natural to look for minimal resolutions among
\textit{categorical} ones. Kuznetsov has given a definition of categorical crepant resolution of singularities which seems very well-fit to deal with the issue of \textit{minimal categorical resolutions of singularities} (see \cite{kuz1}).

\bigskip

The goal of this short note is to prove some ``folklore'' results which concern
categorical crepant resolutions of singularities for varieties with quotient singularities. The first one asserts that  a global quotient always admits a categorical strongly crepant resolution which is non-commutative in the sense of Van den Bergh (see \cite{vdb2} for a definition of non-commutative resolution):

\begin{theo} \label{main1}
Let $V$ be a smooth quasi-projective variety and let $\G$ be a finite subgroup of $\mathrm{Aut}(V)$. Assume that the dualizing sheaf of $V$ is $\G$-equivariantly locally trivial, then $\DU(Coh^{\G}(V))$ is a categorical strongly crepant resolution of $V/ \G$.

Furthermore, there exists a sheaf of algebras $\AA$ on $V/ \G$ such that $\DB(Coh^{\G}(V)) \simeq \DB(V/ \G, \AA)$. Hence, $\DU(Coh^{\G}(V))$ is a non-commutative strongly crepant resolution of $V/{\G}$ in the sense of Van den Bergh.

\end{theo}

\begin{cor} \label{meuh}
Let $V$ be a finite dimensional vector space and let $\G$ be a finite subgroup of $\mathrm{SL}(V)$. Then $X/ \G$ admits a non-commutative strongly crepant resolution in the sense of Van den Bergh.
\end{cor}

Corollary \ref{meuh} is well-known and appeared already many times in the literature (see \cite{vdb2}). A statement analogous to the first part of Theorem \ref{main1} was already foreseen more than 10 years ago, for instance by Bridgeland in \cite{bridge}. Our second result deals with the global case:

\begin{theo} \label{main2}
Let $X$ be a quasi-projective variety with normal Gorenstein quotient singularities and let $\XX$ be the smooth Deligne-Mumford stack whose coarse moduli space is $X$. Assume that the dualizing line bundle of $\XX$ is the pull-back of the dualizing bundle of $X$, then $\DU(\XX)$ is a categorical strongly crepant resolution of $X$.

Furthermore,  there exists a sheaf of algebras $\AA$ on $X$ such that $\DB(\XX) \simeq \DB(X, \AA)$. Hence, $\DU(\XX)$ is a non-commutative strongly crepant resolution of $X$ in the sense of Van den Bergh.
\end{theo}

Theorem \ref{main1} is a consequence of Theorem \ref{main2}. However, I can give a proof of the former which is very low-tech and very direct, whereas the proof of the latter involves some delicate results about Deligne-Mumford stacks. Thus, I believe that it is interesting for the non-expert reader to have both results stated and proved separately.

The two aforementioned results are certainly well-known to experts and I claim no originality for them. Nevertheless, as I could not find any reference, the aim of this note is to put them on a firm ground. As such, I believe that they are already interesting as they provide nice examples of categorical/non-commutative strongly crepant resolutions of singularities. However, from the point of view of \textit{Holomorphically Symplectic Categories} (which will be developed and discussed at length in the forthcoming \cite{abuaf-HKC}), they are of very high importance. Indeed, we have the following corollary of Theorem \ref{main1}:

\begin{cor}
Let $X$ be a smooth projective holomorphically symplectic variety and let $\G$ be a finite subgroup of $\mathrm{Aut}(X)$. Assume that the symplectic form of $X$ is $\G$-equivariant, then $\DB(Coh^{\G}(X))$ is a holomorphically symplectic category.
\end{cor}

Finally, in the last section of this paper, I will compare the strongly crepant resolution discussed above and the derived category of the classical resolution of a variety with scalar cyclic singularities.

\begin{theo}
Let $\G$ be a cyclic group of order $d$ acting on $\mathbb{C}^n$ by translations (with $d$ dividing $n$). Denote by $\widetilde{\mathbb{C}^n/ \G}$ the resolution of singularities of $\mathbb{C}^n/ \G$ obtained by blowing up $0$. There is fully faithful embedding:
\begin{equation*}
\DB(Coh^{\G}(\mathbb{C}^n)) \hookrightarrow \DB(\widetilde{\mathbb{C}^n/ \G}).
\end{equation*}

In particular, if $d=n$, then $\widetilde{\mathbb{C}^n/ \G}$ is a crepant resolution of $\mathbb{C}^n/ \G$ and there is an equivalence:
\begin{equation*}
\DB(Coh^{\G}(\mathbb{C}^n)) \simeq \DB(\widetilde{\mathbb{C}^n/ \G}).
\end{equation*}
\end{theo}
This last result can be seen as an instance of the McKay correspondence for scalar cyclic singularities and is also certainly to be considered as standard (it appears for instance at the $K$-theoretic level in \cite{blume}). As already noticed by Bridgeland, King and Reid, such an equivalence can not be proved using the main result of \cite{BKR}, when $n \geq 4$.

\bigskip

\textbf{Acknowledgments}. I am very grateful to Richard Thomas for mentioning and suggesting to me to prove the aforementioned results on categorical strongly crepant resolutions of quotient singularities. I would also like to thank Tom Bridgeland for sharing his unpublished preprint \cite{bridge} with me, Andrea Petracci for some useful advice on Deligne-Mumford stacks and Enrica Floris for helpful comments on a preliminary version of this paper.

\end{section}

\begin{section}{Categorical crepant resolutions of singularities}
In this section, I remind some basic facts about categorical crepant resolution of singularities and we exhibit some classical examples as they appear in \cite{kuz1}. If $X$ is an algebraic variety over $\mathbb{C}$, we denote by $\DU(X)$ (resp. $\DB(X)$) the unbounded (resp. bounded) derived category of coherent sheaves on $X$.

\begin{defi} Let $X$ be an algebraic variety. A \textit{categorical resolution} of $X$ is a smooth cocomplete compactly generated category $\T$ with a pair of adjoint functors:
\begin{equation*}
\begin{split}
& \pi_* : \T \rightarrow \DU(X) \\
& \pi^* : \DU(X) \rightarrow \T,
\end{split}
\end{equation*}

such that

\begin{itemize}
\item $\pi_* \pi^* \simeq id$,
\item $\pi^*$ and $\pi_*$ commute with arbitrary direct sums,
\item $\pi_*(\T^c) \subset \DB(X)$, where $\T^c$ is the subcategory of compact objects in $\T$.
\end{itemize}
\end{defi}
I will not discuss the details of this definition and rather refer to \cite{kuz2} where the theory is developed with great care. Let me mention one particular example of categorical resolution which is easy to understand.

\begin{prop}
Let $X$ be a quasi-projective variety with rational singularities and let $ q : \X \rightarrow X$ be a resolution of singularities. Let $\T_0$ be an admissible full subcategory of $\DB(\X)$ such that $\LL q^* \DP(X)$ embeds fully and faithfully inside $\T_0$. If $\T$ is the completion of $\T_0$ inside $\DU(\X)$ and $\d : \T \hookrightarrow \DU(\X)$ is the fully faithful embedding, then 

\begin{equation*}\begin{split}
&\RR q_* \circ \delta : \T \rightarrow \DU(X) \\
& \d^* \circ \LL q^* : \DU(X) \rightarrow \T,
\end{split}
\end{equation*}
is a categorical resolution.

\end{prop}

\begin{proof} Since $\X$ is a smooth quasi-projective variety, $\DU(\X)$ is cocomplete, compactly generated (\cite{Neeman}) and smooth (\cite{TV}, lemma $3.27$). But $\T$ is the completion of an orthogonal component of $\DB(\X)$, so that $\T$ is also cocomplete, compactly generated and smooth (\cite{kuz2}, for instance).

 Let me denote by $\pi_*$ the composition $\RR q_* \circ \delta$. We have $\pi_* \pi^* \simeq id$ because $\LL q^* \DP(X)$ embeds fully and faithfully in $\T_0$ and $X$ has rational singularities. The functor $\pi^*$ always commute with arbitrary direct sums whereas the functor $\pi_*$ commute with arbitrary direct sums because it is right adjoint to $\pi^*$ and $\T$ is compactly generated. Finally, the inclusion $\pi_*(\T^c) \subset \DB(X)$ is a consequence of the properness of $q$.

\end{proof}

Note that if $X$ has irrational singularities and $q : \X \rightarrow X$ is a resolution of singularities, then $\RR q_* : \DU(\X) \rightarrow \DU(X)$ is NOT a categorical resolution of singularities. Nevertheless, we still have the following :

\begin{theo}[Kuznetsov-Lunts] 
Any quasi-projective scheme (!) of finite type over $\mathbb{C}$ admits a categorical resolution of singularities.
\end{theo}

I refer to \cite{kuz2} for interesting comments about this nice result and for a proof of it. In this note, we will focus on a very special type of categorical resolutions, the so-called \textit{categorical crepant resolutions} of singularities. These are categorical resolutions which mimick very well the functorial behavior one expects from the derived category of a (geometric) crepant resolution of singularities. The main point of this theory is that we can sometimes construct categorical crepant resolutions even if the singularity has no geometric crepant resolution. This highlights a new point of view on the Minimal Model Program (\cite{kuz1}, \cite{kawa}).

\begin{defi}
Let $X$ be an algebraic variety and let $\pi_* : \T \rightarrow \DU(X)$ be a categorical resolution of $X$. The category $\T$ is said to be a \textbf{weakly crepant} resolution of $Y$ if $\pi^*$ is a also a right adjoint to $\pi_*$ when restricted to $\DP(X)$.

Assume finally that $\T$ has a module structure over $X$. Then $\T$ is said to be a \textbf{strongly crepant} resolution of $X$ if the identity is a relative Serre functor for $\T^c$ with respect to $\DB(X)$.

\end{defi}

The notions of module structure and of relative Serre functor are defined in \cite{kuz1} for instance and they are as natural as one can imagine. Note however that if the categorical resolution is of ``geometric origin'', then all these notions coincide with the classical definition of crepancy. Namely, we have the proposition (\cite{theseabuaf}, prop. $1.2.12$):

\begin{prop} Let $\pi : \X \rightarrow X$ be a morphism of algebraic varieties. Then:
\begin{center}
$\RR \pi_* : \DU(\X) \rightarrow \DB(X)$ is a categorical strongly crepant resolution of singularities 
\end{center}
\begin{center}
$\Longleftrightarrow$ 
\end{center}
\begin{center}
$\RR \pi_* : \DU(\X) \rightarrow \DB(X)$ is a categorical weakly crepant resolution of singularities 
\end{center}
\begin{center}
$\Longleftrightarrow$
\end{center}
\begin{center}
 $\pi : \X \rightarrow X$ is a crepant resolution of singularities.
\end{center}

\end{prop}

The existence of weakly crepant resolutions of singularities has been proved in a quite general context (see \cite{abuaf1, abuaf2}). For instance, it is proved in \cite{abuaf1} that all Gorenstein determinantal varieties (general, symmetric, skew-symmetric) admit weakly crepant resolution of singularities. The existence of strongly crepant resolution seems to be a much more delicate issue. Together with Kuznetsov, we expect the following to be true :

\begin{conj} Let $Y$ be the skew-symmetric determinantal variety  $Y := \{w \in \bigwedge^2 \mathbb{C}^n, \, \text{such that} \, \, \mathrm{rk} w \leq 1 \}$. Then $Y$ admits a strongly crepant resolution if and only if $n$ is odd.
\end{conj}

It is proved in $\cite{kuz1}$ that for odd $n$, such cones indeed admit strongly crepant resolution. As for even $n$, only weakly crepant resolutions could be constructed.

\bigskip

In \cite{kuz1}, Kuznetsov proved the existence of categorical strongly crepant resolutions of singularities for scalar cyclic quotient singularities. Let me summarize how he found them, as I show in the last section of this paper that they are equivalent to the general construction I provide for quotient singularities. Recall that a (local) scalar cyclic quotient singularity is the data of a cyclic group $\G$ which acts by translation on $\mathbb{C}^n$. It is well-known that the quotient $\mathbb{C}^n / \G$ is Gorenstein if and only if the order of $\G$ (say $d$) divides $n$ and that $\mathbb{C}^n / \G$ admits a crepant resolution of singularities if and only if $d=n$ (see \cite{Fujiki}). In the following, I always assume that $d$ divides $n$.

Let me denote $X := \mathbb{C}^n / \G$ and $q : \X \rightarrow X$ the blow-up of $X$ at $0$. The morphism $q$ is a resolution of singularities and the exceptional divisor $i :E \hookrightarrow \X$ is isomorphic to $\mathbb{P}^{n-1}$ with $\OO_{E}(E) = \OO_{\mathbb{P}^{n-1}}(-d)$. Note that $\X$ is the total space of $ t : \OO_{\mathbb{P}^{n-1}}(-d) \rightarrow \mathbb{P}^{n-1}$.

\begin{theo}[\cite{kuz1}] \label{kuznetsov}
There is a semi-orthogonal decomposition:

\begin{equation*}
\DB(\X) = \langle i_* \left( \BB \otimes \OO_E((\frac{n}{d} -1)E) \right), \cdots,  i_* \left( \BB \otimes \OO_E(E) \right), \T_0 \rangle,
\end{equation*}
where $\BB = \langle \OO_{\mathbb{P}^{n-1}}((1-d)), \cdots, \OO_{\mathbb{P}^{n-1}}(-1), \OO_{\mathbb{P}^{n-1}} \rangle$.

The completion of $\T_0$ in $\DU(\X)$ is a categorical strongly crepant resolution of $X$ and there is an equivalence:

\begin{equation*}
\T_0 \simeq \DB(X, q_* \mathcal{E}nd(\AA)),
\end{equation*}
where $\AA = t^* \OO_{\mathbb{P}^{n-1}} \oplus t^*\OO_{\mathbb{P}^{n-1}}(-1) \oplus \cdots \oplus t^*\OO_{\mathbb{P}^{n-1}}(1-d)$.

\end{theo}

\end{section}

\begin{section}{Main results and proofs}
In this section, I will prove Theorems \ref{main1} and \ref{main2} and some related corollaries.

\begin{theo} 
Let $V$ be a smooth quasi-projective variety and let $\G$ be a finite subgroup of $\mathrm{Aut}(V)$. Assume that the dualizing sheaf of $V$ (denoted $\omega_V$) is $\G$-equivariantly locally trivial, then $\DU(Coh^{\G}(V))$ is a categorical strongly crepant resolution of $V/ \G$.

Furthermore, there exists a sheaf of algebras $\AA$ on $X/ \G$ such that $\DB(Coh^{\G}(V)) \simeq \DB(V/ \G, \AA)$. Hence, $\DU(Coh^{\G}(V))$ is a non-commutative strongly crepant resolution of $X/{\G}$ in the sense of Van den Bergh.

\end{theo}

The dualizing line bundle of $V$ being $\G$-equivariantly trivial means that for each $v \in V$, one can find a $\G$-stable open neighborhood of $v \in V$ such that $\omega_V$ restricted to that open is $\G$-equivariantly trivial. This implies in particular that if $v \in V$ is a fixed point of the action on $\G$ on $V$, then the representation of $\G_v$ on $\omega_{V,v}$ is trivial. Hence, by descent theory for quotient maps by finite groups, we find that $V/ \G$ is Gorenstein and that $\omega_V$ is the pull-back of $\omega_{V/ \G}$ (see \cite{peskin}).

\begin{proof} The category $\DU(Coh^{\G}(V))$ is obviously cocomplete. Since $V$ is quasi-projective, it admits a $\G$-equivariant ample line bundle, say $L$. The same assertions as in \cite{Neeman}, example $1.10$, carry over to show that if $V_{\rho_1}, \cdots, V_{\rho_p}$ are the finite irreducible representations of $\G$, then the family:

\begin{equation*}
 \{ L^{\otimes m}[n] \otimes V_{\rho_l} \}^{l=1 \cdots p}_{n, m \in \mathbb{Z}}
\end{equation*}
 generates $\DU(Coh^{\G}(V))$. Finally, since $V$ is smooth and $\DU(Coh^{\G}(V))$ is compactly generated, the same argument as in (\cite{TV}, Lemma $3.27$) show that $\DU(Coh^{\G}(V))$ is smooth.

\bigskip

Denote by $\pi : V \rightarrow V/G$ the quotient map. There is a functor:

\begin{equation*}
\pi_*^{\G} : \DU(Coh^{\G}(V)) \rightarrow \DU(V/ \G),
\end{equation*}
where $\FF \rightarrow \FF^{\G}$ is the functor of $\G$-invariants. The left adjoint to $\pi_*^{\G}$ is $\LL \pi^*$, where the $\G$ action on $\LL \pi^* \FF$ for any $\FF$ in $\DU(V/ \G)$ is the trivial one.

 By construction of the quotient $V/ \G$, we have $\pi_*^{\G} \OO_V = \OO_{V/ \G}$. The projection formula implies that:

\begin{equation*}
\pi_*^{\G} \LL \pi^* \simeq id.
\end{equation*}
The functor $\LL \pi^*$ always commutes with arbitrary direct sums, whereas $\pi_*^{\G}$ commutes with arbitrary direct sums because it is right adjoint to $\LL \pi^*$ ans $\DU(Coh^{\G}(V))$ is compactly generated. The inclusion $\pi_*^{\G} \DB(Coh^{\G}(V)) \subset \DB(V/ \G)$ comes from the properness of $\pi$.

\bigskip

Let me now study the right adjoint to $\LL \pi^*$. The dualizing line bundle of $V$ is $\G$-equivariantly locally trivial, hence $\omega_{V} = \pi^* \omega_{V/ \G}$. Thus, by Grothendieck duality, we have:

\begin{equation*}
\Hh_{\DB(V/ \G)}( \pi_* \GG, \FF) = \Hh_{\DB(V)}( \GG, \LL \pi^* \FF),
\end{equation*} 
for all $\GG \in \DB(V)$ and all $\FF \in \DP(V/ \G)$. As a consequence, for any $\GG \in \DB(V)$, we have:

\begin{equation*}
\begin{split}
\pi_* \RR \H om(\GG, \OO_V) & = \pi_* \RR \H om(\GG, \LL \pi^* \OO_{V/ \G})\\
                                                & = \RR \H om( \pi_* \GG, \OO_{V/ \G}).\\
\end{split}
\end{equation*}
If $\GG$ is assumed to be $\G$-equivariant, then:
\begin{equation*}
\RR \H om(\pi_* \GG, \OO_{V / \G})^{\G} = \RR \H om(\pi_*^{\G} \GG, \OO_{V / \G}).
\end{equation*}

From the above equalities, we find:
\begin{equation*}
\begin{split}
\Hh_{\DB(V / \G)}( \pi_*^{\G} \GG, \FF) & =  \Hh_{\DB(V / \G)}( \OO_{V / \G}, \RR \H om (\pi_*^{\G} \GG, \OO_{V/ \G}) \otimes\FF)\\
                                                                  & = \Hh_{\DB(V / \G)}( \OO_{V / \G}, \pi_*^{\G}\RR \H om (\GG, \OO_{V}) \otimes\FF) \\
                                                                  & = \Hh_{\DB(V / \G)}( \OO_{V / \G}, \pi_*^{\G} \left(\RR \H om (\GG, \OO_{V}) \otimes\LL \pi^*\FF \right)) \\
                                                                   & =  \Hh_{\DB(Coh^{\G}(V)}( \OO_{V}, \RR \H om (\GG, \OO_{V}) \otimes\LL \pi^*\FF )\\
                                                                   & = \Hh_{\DB(Coh^{\G}(V))}( \GG, \LL \pi^* \FF)
\end{split}
\end{equation*} 
for all $\GG \in \DB(Coh^{\G}(V))$ and all $\FF \in \DP(V/ \G)$. Hence, $\LL \pi^*$ is a right adjoint to $\pi_*^{\G}$, well-defined on $\DP(V/ \G)$. This already proves that $ \pi_*^{\G}: \DU(Coh^{\G}(V)) \rightarrow \DU(V / \G)$ is a categorical weakly crepant resolution of singularities of $V / \G$. I go on demonstrating that it is in fact a strongly crepant resolution of singularities.

\bigskip

The category $\DU(Coh^{\G}(V))$ has obviously a $V / \G$ module structure as it is naturally endowed with a tensor product. So I am only left to prove that the identity is a relative Serre functor for $\DB(Coh^{\G}(V))$ with respect to $\DB(V / \G)$. Let $\GG_1$ and $\GG_2$ be in $\DB(Coh^{\G}(V))$, we have:

\begin{equation*}
\begin{split}
\RR \pi_*^{\G} \RR \H om(\GG_1, \GG_2) & = \RR \pi_*^{\G} \RR \H om(\GG_1 \otimes \RR \H om(\GG_2, \OO_V ), \OO_V)\\
                                                            & =  \RR \pi_*^{\G} \RR \H om(\GG_1 \otimes \RR \H om(\GG_2, \OO_V ), \pi^{!} \OO_{V/ \G})\\
                                                            & =  \RR \H om(\RR \pi_*^{\G}(\GG_1 \otimes \RR \H om(\GG_2, \OO_V )), \OO_{V/ \G})\\
                                                            & = \RR \H om(\RR \pi_*^{\G}(\RR \H om(\GG_2, \GG_1 )), \OO_{V/ \G}).
\end{split}
\end{equation*}
This shows that the identity is indeed a relative Serre functor  for $\DB(Coh^{\G}(V))$ with respect to $\DB(V / \G)$.

\bigskip

I want to prove that to prove that  $\DU(Coh^{\G}(V))$ is non-commutative in the sense of Van den Bergh, i.e. there exists a sheaf of algebras $\AA$ on $V/ \G$ such that:
\begin{equation*}
 \DB(Coh^{\G}(V)) \simeq \DB(V/G, \AA).
\end{equation*}
 Let $V_{\rho_1}, \cdots, V_{\rho_m}$ be all the finite irreducible representations of $\G$. Let's prove that:
\begin{equation*}
 \GG := \OO_V \otimes V_{\rho_1} \oplus \cdots \oplus \OO_V \otimes V_{\rho_m}
\end{equation*}
 is a tilting bundle for $\DB(Coh^{\G}(V))$ with respect to $\pi_*^{\G}$. Notice that the functor $\pi_*^{\G}$ is exact, so that I only have to check that for any equivariant coherent sheaf $\FF$, the vanishing $\pi_*^{\G} \H om(\GG, \FF) = 0$ implies $\FF =0$.

Let $\FF$ be $\G$-equivariant sheaf such that  $\pi_*^{\G} \H om(\GG, \FF) = 0$. The morphism $\pi$ is flat, the functor $\pi_*$ is exact and $\GG$ is locally free, hence by base change, this is equivalent to:

\begin{equation*}
 \H om(\GG, \FF \otimes \OO_{\pi^{-1}(y)})^{\G} = 0,
\end{equation*}
for all $y \in V/ \G$. We have a $\G$-equivariant surjection $\GG \rightarrow \GG \otimes \OO_{\pi^{-1}(y)}$, so that the above vanishing implies:

\begin{equation*}
 \H om(\GG \otimes \OO_{\pi^{-1}(y)} , \FF \otimes \OO_{\pi^{-1}(y)})^{\G} = 0.
\end{equation*}
But $\OO_{\pi^{-1}(y)}$ is a finite $\mathbb{C}$-algebra and $\FF$ is a $\G$-equivariant coherent sheaf. Hence, by Schur's lemma, we find $\FF \otimes \OO_{\pi^{-1}(y)} = 0$ for all $y \in V/ \G$, that is $\FF = 0$. I thus proved that $\GG$ is a tilting bundle for $\DB(Coh^{\G}(V))$ with respect to $\pi_*^{\G}$. We deduce that (see \cite{HVDB}):

\begin{equation*}
\DB(Coh^{\G}(V)) \simeq \DB(V / \G, \pi_*^{\G} \mathcal{E}nd(\GG)).
\end{equation*}

\end{proof}

\begin{cor}
Let $V$ be a finite dimensional vector space and let $\G$ be a finite subgroup of $\mathrm{SL}(V)$. Then $V/ \G$ admits a non-commutative strongly crepant resolution in the sense of Van den Bergh.
\end{cor}

This corollary already appeared many times in the literature (see \cite{vdb2}). Note that our construction precisely gives the expected algebra for the non-commutative resolution of $V / \G$, that is the skew-algebra $\mathrm{Sym} V \# G$.

\begin{proof}
Indeed, if $\G$ is a subgroup of $\mathrm{SL}(V)$, then $\G$ acts trivially on the volume form of $V$. This implies that $\omega_V$ is $\G$-equivariantly trivial, so that we can apply the above result.
\end{proof}

\begin{cor}
Let $X$ be a smooth projective holomorphically symplectic variety and let $\G$ be a finite subgroup of $\mathrm{Aut}(X)$. Assume that the symplectic form of $X$ is $\G$-equivariant, then $\DB(Coh^{\G}(X))$ is a holomorphically symplectic category.
\end{cor} 

This last corollary will proved and discussed with great care in the forthcoming \cite{abuaf-HKC}. I now deal with my second main result:

\begin{theo}
Let $X$ be a quasi-projective variety with normal Gorenstein quotient singularities and let $\XX$ be the smooth separated Deligne-Mumford stack whose coarse moduli space is $X$. Assume that the dualizing line bundle of $\XX$ is the pull back of the dualizing line bundle on $X$, then $\DU(\XX)$ is a strongly crepant resolution of $X$.

Furthermore,  there exists a sheaf of algebras $\AA$ on $X$ such that $\DB(\XX) \simeq \DB(X, \AA)$. Hence, $\DU(\XX)$ is a non-commutative strongly crepant resolution of $X$ in the sense of Van den Bergh.

\end{theo}

Note that if $X$ is a normal quasi-projective variety with quotient singularities, then there is always a smooth separated Deligne-Mumford stack associated to it as in the above statement (see proposition $2.8$ of \cite{vistoli}). The non-trivial hypothesis (which can not be removed) is that the dualizing bundle of the Deligne-Mumford stack associated to $X$ is the pull back of the dualizing bundle on $X$. This amounts to check that on an \'etale atlas of $\XX$, the line bundle $\omega_{\XX}$ is equivariantly \footnote{for the isotropy groups of the fixed points of the \'etale atlas of $\XX$} locally trivial. This condition is probably not obvious to check in general, but  I believe that a very precise description of the singularities of $X$ might help one to decide whether it is satisfied or not.

\begin{proof}
This proof follows exactly the same pattern as the proof of Theorem \ref{main1}. However, in the present case, I have to use some delicate results about Deligne-Mumford stacks in order to have the same machinery work.

\bigskip

The category $\DB(\XX)$ is certainly cocomplete. By \cite{Toen}, corollary $5.2$, the category $\DU(\XX)$ is also compactly generated. Since $\XX$ is a smooth and separated Deligne-Mumford stack and $\DU(\XX)$ is compactly generated, a small variation on lemma $3.27$ in \cite{TV} shows that $\DU(\XX)$ is smooth.

Let $\pi : \XX \rightarrow X$ be the projection from $\XX$ to its coarse moduli space. By construction of the coarse moduli space of a Deligne-Mumford stack, we know that $\pi_*$ is exact and that $\pi_* \OO_{\XX} = \OO_X$. Hence, by the projection formula, we have:

\begin{equation*}
\pi_* \LL \pi^* \simeq id,
\end{equation*}

where $\LL \pi^*$ is the left adjoint to $\pi_*$. The functor $\LL \pi^*$ always commutes with arbitrary direct sums, whereas $\pi_*$ commutes with arbitrary direct sums because it is right-adjoint to $\LL \pi^*$ and $\DB(\XX)$ is compactly generated. Finally, the morphism $\pi$ is proper, so that $\pi_* \DB(\XX) \subset \DB(X)$.

\bigskip

 Let me show that the relative Serre functor of $\DB(\XX)$ with respect to $\DB(X)$ is trivial. By Grothendieck duality for Deligne-Mumford stack (see \cite{nironi}, Theorem $2.27$ or \cite{yeku}), the right adjoint to $\pi_*$ is equal to $\LL \pi^* \otimes \omega_{\XX} \otimes \pi^* \omega_X^{-1}$ when restricted to $\DP(X)$. This already proves that $\DU(\XX)$ is a categorical weakly crepant resolution of $X$. But the category $\DU(\XX)$ is endowed with a tensor product, so that the same basic computations as in the proof of Theorem \ref{main1} show that the relative Serre functor $\DB(\XX)$ with respect to $\DB(X)$ is indeed trivial.

\bigskip

I am left to prove that the exists a tilting vector bundle for $\DB(\XX)$ with respect to $\pi$. But $\pi_*$ is exact, so I only have to prove the existence of a vector bundle $\GG$ on $\XX$ such that for any coherent sheaf $\FF$ on $\XX$, the vanishing $\pi_* \H om(\GG, \FF) = 0$ implies $\FF =0$. The existence of such a vector bundle is exactly the content Theorem $4.4$ and Theorem $5.3$ of \cite{kresch}.

\end{proof}

\end{section}

\begin{section}{Some connections with the McKay correspondence}
I have proved that for any smooth quasi-projective variety $V$ acted on by a finite group $\G$ such that $\omega_{V}$ is $\G$-equivariantly locally trivial, the derived category $\DU(Coh^{\G}(V))$ is a strongly crepant resolution of $V/ \G$. In this situation, Conjecture $4.10$ of \cite{kuz1} can be restated as follows:

\begin{conj} \label{conjconj}
Let $V$ be a smooth quasi-projective variety and let $\G$ be a finite subgroup of $\mathrm{Aut}(V)$ such that $\omega_V$ is $\G$-equivariantly locally trivial. Then, for any resolution of singularities $Z \rightarrow V/ \G$, there exists an admissible fully faithful embedding:

\begin{equation*}
\DB(Coh^{\G}(V)) \hookrightarrow \DB(Z).
\end{equation*}
In particular, if $Z$ is a crepant resolution of $V/ \G$, then there is an equivalence:

\begin{equation*}
\DB(Coh^{\G}(V)) \simeq \DB(Z).
\end{equation*}
\end{conj}

The last part of this conjecture has been proved in \cite{BKR} if $\dim V \leq 3$ or if $V$ is symplectic and $\G$ acts by symplectic automorphisms. This result is known as the categorical McKay correspondence. We will prove a very special case of the conjecture in the context of cyclic groups acting by translations on finite dimensional vector spaces.

\begin{theo} \label{main3}
Let $V$ be a vector space of dimension $n$ and let $\G$ be a cyclic group of order $d$ acting by translations on $V$ (with $d$ dividing $n$). Then we have an equivalence:

\begin{equation*}
\DB(Coh^{\G}(V)) \simeq \T_0,
\end{equation*}
where $\T_0$ is the categorical strongly crepant resolution of $V/ \G$ constructed by Kuznetsov (see Theorem \ref{kuznetsov}). In particular, there is a fully faithful embedding:
\begin{equation*}
\DB(Coh^{\G}(V)) \hookrightarrow \DB(\widetilde{V/ \G}),
\end{equation*}
where $\widetilde{V/ \G}$ is the resolution of $V/ \G$ obtained by blowing-up $0$. In the special case where $d=n$ (so that $\widetilde{V/ \G}$ is a crepant resolution of $V / \G$), there is an equivalence:
\begin{equation*}
\DB(Coh^{\G}(V)) \simeq \DB(\widetilde{V/ \G}),
\end{equation*}

\end{theo}

If $\G$ is a cyclic group acting on $V$ by translations, the blow-up of $V / \G$ along $0$ is the ``smallest'' geometric resolution of singularities known for $V / \G$. Hence, in that specific case, I believe that Theorem \ref{main3} should render a proof of Conjecture \ref{conjconj} tractable.

\begin{proof}

We denote $X := V/ \G$ and $q : \X \rightarrow X$ the blow-up of $X$ along $0$. Let $q : \tilde{V} \rightarrow V$ the blow-up of along $0$ and consider the projection $p_*^{G} : \DB(Coh^{\G}(\V) \rightarrow \DB(\X)$, where $p : \V \rightarrow \X$ is the quotient map and $\FF \rightarrow \FF^{\G}$ is the functor of invariants.

As $V$ is smooth, the category $\DB(Coh^{\G}(V))$ is a full admissible subcategory of $\DB(Coh^{\G}(\V)$. Let $\chi_1, \cdots, \chi_{d-1}$ be the characters of $\G$. The same argument as in the proof of Theorem \ref{main1} shows that the vector bundle:

\begin{equation*}
 \OO_{\V} \otimes \chi_0 \oplus \cdots \oplus \OO_{\V} \otimes \chi_{d-1}
\end{equation*}

is a tilting bundle for $\DB(Coh^{\G}(V))$ with respect to $p_*^{\G}$. Recall that $\X$ is the total space of $ t : \OO_{\mathbb{P}^{n-1}}(-d) \rightarrow \mathbb{P}^{n-1}$. The quotient map $p : \V \rightarrow \X$ is a $d$ to $1$ cover ramified along the zero section of $t$. Hence, we have:

\begin{equation*}
p_* \OO_{\V} = t^* \OO_{\mathbb{P}^{n-1}}(-d+1) \oplus \cdots \oplus t^* \OO_{\mathbb{P}^{n-1}}.
\end{equation*}

In particular:

\begin{equation*}
p_*^{\G} (\OO_{\V} \otimes \chi_j) = t^* \OO_{\mathbb{P}^{n-1}}(-j).
\end{equation*}

Since we have the vanishing $\RR^i q_* t^* \OO_{\mathbb{P}^{n-1}}(-j)$ for all $1-d \leq j \leq d-1$ and all $i >0$, we deduce that $\OO_{\V} \otimes \chi_0 \oplus \cdots \oplus \OO_{\V} \otimes \chi_{d-1}$ is a tilting bundle for $\DB(Coh^{\G}(V))$ with respect to $q_* p_*^{\G}$. As a consequence, we have:

\begin{equation*}
\DB(Coh^{\G}(V)) \simeq \DB(X, q_* p_*^{\G} \left( \mathcal{E}nd ( \OO_{\V} \otimes \chi_0 \oplus \cdots \oplus \OO_{\V} \otimes \chi_{d-1}) \right)).
\end{equation*}

From the computations just above, we immediately find that:
\begin{equation*}
 q_* p_*^{\G} \left( \mathcal{E}nd ( \OO_{\V} \otimes \chi_0 \oplus \cdots \oplus \OO_{\V} \otimes \chi_{d-1}) \right) = q_* \mathcal{E}nd( t^* \OO_{\mathbb{P}^{n-1}}(-d+1) \oplus \cdots \oplus t^* \OO_{\mathbb{P}^{n-1}}).
\end{equation*}

We then deduce that:
\begin{equation*}
\DB(Coh^{\G}(V)) \simeq \T_0,
\end{equation*}

where $\T_0$ is the categorical strongly crepant resolution of $X$ constructed by Kuznetsov (see Theorem \ref{kuznetsov}).

\bigskip

If $d=n$, one notices that $\T_0 = \DB(\X)$, so that $\DB(Coh^{\G}(V)) \simeq \DB(\X)$.

\end{proof}

\bigskip

Note that the rough idea for the proof of Theorem \ref{main3} is already present in \cite{bridge}.

\end{section}

\bibliographystyle{alpha}

\bibliography{bibliHKC}

\end{document}